\documentclass[12pt]{amsart}
\usepackage{amscd,amsmath,amsthm,amssymb,graphics}

\usepackage{amsmath}
\usepackage{amssymb}
\usepackage{amsbsy}
\usepackage{graphicx}
\usepackage{amsfonts}

\usepackage{amscd,amsmath,amsthm,amssymb}
\usepackage{pstcol,pst-plot,pst-3d}
\usepackage{lmodern,pst-node}

\usepackage{pgf,tikz}
\usetikzlibrary{arrows}

\def\frk{\frak}               

\def\Phi{{\frk n}}
\def\Phi{{\frk N}}
\def\opn#1#2{\def#1{\operatorname{#2}}} 
\opn\chara{char} \opn\length{\ell} \opn\pd{pd} \opn\rk{rk}
\opn\projdim{proj\,dim} \opn\injdim{inj\,dim} \opn\rank{rank}
\opn\depth{depth} \opn\grade{grade} \opn\height{height}
\opn\embdim{emb\,dim} \opn\codim{codim}

\opn\Tr{Tr} \opn\bigrank{big\,rank}
\opn\superheight{superheight}\opn\lcm{lcm}
\opn\trdeg{tr\,deg}
\opn\reg{reg} \opn\lreg{lreg} \opn\ini{in} \opn\lpd{lpd}
\opn\size{size}\opn\bigsize{bigsize}
\opn\cosize{cosize}\opn\bigcosize{bigcosize}
\opn\sdepth{sdepth}\opn\sreg{sreg}
\opn\link{link}\opn\fdepth{fdepth}
\opn\index{index}
\opn\index{index}
\opn\indeg{indeg}
\opn\N{N}
\opn\SSC{SSC}
\opn\SC{SC}
\opn\conv{conv}
\opn\div{div} \opn\Div{Div} \opn\cl{cl} \opn\Cl{Cl}
\opn\Spec{Spec} \opn\Supp{Supp} \opn\supp{supp} \opn\Sing{Sing}
\opn\Ass{Ass} \opn\Min{Min}\opn\Mon{Mon} \opn\dstab{dstab} \opn\astab{astab}
\opn\Syz{Syz}
\opn\reg{reg}
\opn\Ann{Ann} \opn\Rad{Rad} \opn\Soc{Soc}
\opn\Im{Im} \opn\Ker{Ker} \opn\Coker{Coker} \opn\Am{Am}
\opn\Hom{Hom} \opn\Tor{Tor} \opn\Ext{Ext} \opn\End{End}
\opn\Aut{Aut} \opn\id{id}

\opn\nat{nat}
\opn\pff{pf}
\opn\Pf{Pf} \opn\GL{GL} \opn\SL{SL} \opn\mod{mod} \opn\ord{ord}
\opn\Gin{Gin} \opn\Hilb{Hilb}\opn\sort{sort}
\opn\initial{init}
\opn\ende{end}
\opn\height{height}
\opn\type{type}
\opn\aff{aff} \opn\con{conv} \opn\relint{relint} \opn\st{st}
\opn\lk{lk} \opn\cn{cn} \opn\core{core} \opn\vol{vol}
\opn\link{link} \opn\star{star}\opn\lex{lex}\opn\Mon{Mon}\opn\Min{Min}
\opn\gr{gr}

\def\pot#1#2{#1[\kern-0.28ex[#2]\kern-0.28ex]}

\opn\dirlim{\underrightarrow{\lim}}
\opn\inivlim{\underleftarrow{\lim}}
\def\Implies{\ifmmode\Longrightarrow \else
        \unskip${}\Longrightarrow{}$\ignorespaces\fi}
\def\implies{\ifmmode\Rightarrow \else
        \unskip${}\Rightarrow{}$\ignorespaces\fi}
\def\iff{\ifmmode\Longleftrightarrow \else
        \unskip${}\Longleftrightarrow{}$\ignorespaces\fi}

\let\:=\colon
\newtheorem{Theorem}{Theorem}[section]
 
 \newtheorem{Corollary}[Theorem]{Corollary}
 \newtheorem{Proposition}[Theorem]{Proposition}
 \newtheorem{Remark}[Theorem]{Remark}
 
 \newtheorem{Example}[Theorem]{Example}
 
 \newtheorem{Definition}[Theorem]{Definition}
 \newtheorem*{Definition*}{Definition}
 
 \newtheorem{Conjecture}{Conjecture}
 \newtheorem*{Conjecture*}{Conjecture}

\let\epsilon\varepsilon
\let\kappa=\varkappa
\def\qed{\ifhmode\textqed\fi
      \ifmmode\ifinner\quad\qedsymbol\else\dispqed\fi\fi}
\def\textqed{\unskip\nobreak\penalty50
       \hskip2em\hbox{}\nobreak\hfil\qedsymbol
       \parfillskip=0pt \finalhyphendemerits=0}
\def\dispqed{\rlap{\qquad\qedsymbol}}

\opn\dis{dis}
\def\pnt{{\raise0.5mm\hbox{\large\bf.}}}

\opn\Lex{Lex}

\begin{document}

 \title{On  powers of  cover ideals of  graphs}

 \author {Dancheng Lu and Zexin Wang}

 \begin{abstract} For a simple graph $G$, assume that $J(G)$ is the vertex cover ideal of $G$ and $J(G)^{(s)}$ is the $s$-th symbolic power of $J(G)$. We  prove that $\reg(J(C)^{(s)})=\reg(J(C)^s)$ for all $s\geq 1$  and for all odd cycle $C$. For a simplicial complex $\Delta$, we show that  if $I_{\Delta}^{\vee}$ is weakly polymatroidal (not necessarily generated in one degree) then  $\Delta$ is vertex decomposable. Some evidences are provided  that the converse conclusion of the above result also holds true if $\Delta$ is pure. Let $W=G^{\pi}$ be a fully clique-whiskering graph. We prove that $J(W)^s$ is weakly polymatroidal for all $s\geq 1$.
 \end{abstract}

\address{Dancheng Lu, School of Mathematical Science, Soochow University, P.R.China} \email{ludancheng@suda.edu.cn}
\address{Zexin Wang,  School of Mathematical Science, Soochow University, P.R.China} \email{zexinwang6@outlook.com}

 \maketitle

\section{Introduction}

Let $R=K[x_1,\ldots,x_n]$ be the polynomial ring over a field $K$ and let $G$ be a simple graph on vertex set $[n]:=\{1,2,\ldots,n\}$ with edge set $E(G)$. There are two square-free monomial ideals of $R$ associated to $G$: the {\it edge ideal} $I(G)$ which is generated by all monomial $x_ix_j$ with $\{i,j\}\in E(G)$ and the {\it vertex cover ideal} $J(G)$ generated by monomials $\prod_{i\in F}x_i$, where $F$ is taken over all  minimal vertex covers of $G$. Recall that  a subset $F$ of $V(G)$ is a {\it vertex cover} of $G$ if $F\cap e\neq \emptyset$ for every edge $e$ of $G$ and a vertex cover $F$ of $G$ is {\it minimal} if $F\setminus \{i\}$ is not a vertex cover for each $i\in F$. The vertex cover ideal $J(G)$ is the Alexander dual of the edge ideal $I(G)$, i.e., $$J(G)=I(G)^{\vee}=\bigcap_{\{i,j\}\in E(G)}(x_i,x_j).$$

Let $I$ be a graded ideal of $R$. The $s${\it -th symbolic power} of $I$ is defined by $$I^{(s)}=\bigcap_{\mathfrak{p}\in \mathrm{Min}(I)}I^sR_{\mathfrak{p}}\cap R,$$
where $\mathrm{Min}(I)$ is as usual the set of all minimal prime ideals of $I$.  It follows from  \cite[Proposition 1.4.4]{HH} that for every integer $s\geq 1$,  $$J(G)^{(s)}=\bigcap_{\{i,j\}\in E(G)}(x_i,x_j)^s.$$

The Castelnuovo-Mumford regularity (or simply regularity) is a fundamental invariant in commutative
algebra and algebraic geometry. For a  finitely generated graded module $M$ over the polynomial ring $R$, the {\it regularity} of $M$, denoted by $\reg(M)$,  is the least integer $r\geq 0$ such that for all $i\geq 0$,  the $i$-th syzygy of $M$ is generated by homogeneous elements of degree at most $r+i$. An equivalent definition of the regularity via local cohomology is as follows:$$\reg(M)=\max\{i+j\:\; H^i_{\mathfrak{m}}(M)_j\neq 0, i\geq 0, j\in \mathbb{Z}\}.$$ Here $\mathfrak{m}$ denotes the maximal  ideal $(x_1,\ldots,x_n)$.

For edge ideals of graphs, there have been a lot of research on connections between
the regularity functions $\reg(I(G)^s)$ as well as $\reg(I(G)^{(s)})$ and  the combinatorial properties of $G$, see \cite{BBH} and the references therein. Recently the conjecture that  $\reg(I(G)^{(s)})= \reg(I(G)^s)$  for all $s\geq 1$ and for all graphs $G$ attracted much attention and much progress has been made in this direction. For the details, see  \cite{JK} and the references therein.

 Meanwhile, the study of algebraic properties of (symbolic  and ordinary) powers of vertex cover ideals of graphs is also an active research topic. However, the regularity of  powers of such ideals is harder to compute or deal with. In fact, although S.A. Seyed Fakhari presented in \cite{S} the following remarkable bounds for a large class of graphs $G$, including bipartite graphs, unmixed graphs, claw-free graphs: $$s\mathrm{Deg}(J(G))\leq \mathrm{reg}(J(G)^{(s)})\leq (s-1)\mathrm{Deg}(J(G))+|V(G)|-1,$$ there are not many graphs $G$ for which either $\reg(J(G)^{(s)})$ or $\reg(J(G)^{s})$ is known precisely.  Here, $\mathrm{Deg}(J(G))$ is the maximum size of minimal vertex covers of $G$.   When $G$ is  either a crown graph or a complete multipartite graph,   $\reg(J(G)^{(s)})$  was explicitly given in \cite{KKSS}. On the other side,  if a graded ideal is componentwise linear then its regularity is equal to the maximum degree of its minimal generators. In the literatures \cite{DHN,GHS,M,S,Sel} and \cite{S1},  some classes of graphs for which either $J(G)^s$ or $J(G)^{(s)}$ is componentwise linear are identified. For examples,  it was proved in \cite{S1} that $G$  is a Cohen-Macaulay very well-covered graph if and only if   $J(G)^{(s)}$ has a linear resolution for some (equivalently, for all ) integer $s\geq 2$. In \cite{DHN, GHS, Sel,S} and \cite{S1}, among others, they investigate the question of how to combinatorially modify a graph  to obtain componentwise linearity of the corresponding monomial ideals,  and identify  many graphs $G$ such that $J(G)^{(s)}$ is componentwise linear. In \cite{M}, it was proved that if $G$ is a Cohen-Macaulay cactus graph then $J(G)^s$  has a linear resolution for all $s\geq 1$. For such graphs $G$, the regularity of   either $J(G)^{s}$ or $J(G)^{(s)}$ is known. More recently,  a new upper  bound for $\mathrm{reg}(J(G)^{(s)})$  was given in \cite{N} when $G$ is a non-bipartite graph.

In this paper we investigate further the  properties of (symbolic and ordinary) powers of vertex cover ideals of simple graphs. Our first main result is motivated  by Theorem 5.15 in \cite{DHN}, in which a family of graphs $G$ was constructed such that $\reg(J(G)^{(s)})$ is not eventually linear in $s$. This result particularly shows that the equality  $\reg(J(G)^{(s)})=\reg(J(G)^{s})$ is not true in general.  On the other side, we have known the equality  $\reg(I(G)^{(s)}= \reg(I(G)^s)$  for all $s\geq 1$ holds for many classes of graphs such as unicyclic graphs and chordal graphs and so on. These facts lead us to ask the following question:

\begin{center} {\it For which graphs $G$ has one   $\reg(J(G)^{(s)})=\reg(J(G)^{s})$  for all $s\geq 1$?}
\end{center}

   In this vein we prove the following result.

\vspace{1.5mm}
   \noindent {\bf Theorem \ref{main1}} {\it If $C$ is an odd cycle, then $\reg(J(C)^{(s)})=\reg(J(C)^{s})$ for all $s\geq 1$.}
   \vspace{1.5mm}

  We remark this  is the first non-trivial example where  the above formula holds, due to the well-known fact that  $G$ is a bipartite graph if and only if  $J(G)^{(s)}=J(G)^s$ for some $s\geq 2$ (equivalently for all $s\geq 1$), see \cite[Proposition 1.3]{TT} or  \cite[Theorem 5.1]{HHT}.

\vspace{3mm}
 We next  investigate  relations between weak polymatroidality of a monomial ideal and vertex decomposability of a simplicial complex. It turns out it helps to understand the  behaviors of the powers of cover ideals. Vertex decomposability  was first introduced by \cite{PB} in
the pure case, and extended to the non-pure case in \cite{BW}.
It is defined in terms of the deletion and link. Let $\Delta$ be a simplicial complex on vertex set $V$.  For  $x\in V$, the {\it link} of $x$ in $\Delta$  is the subcomplex $$\mbox{lk}_{\Delta}(x)=\{F\in \Delta\:\; F\cup \{x\}\in \Delta \mbox{ and } x\notin F\}$$ and the {\it deletion}  of $x$ in $\Delta$  is the subcomplex $$\mbox{del}_{\Delta}(x)=\{F\in \Delta\:\;  x\notin F\}.$$
\begin{Definition}\em  A simplicial complex  $\Delta$ is said to be {\it vertex decomposable} if either $\Delta$ is a simplex, or  there exists a vertex $x$ of $\Delta$ such that

(1) $\mbox{lk}_{\Delta}(x)$ and $\mbox{del}_{\Delta}(x)$ are vertex decomposable;

(2) Each facet of $\mbox{lk}_{\Delta}(x)$ is not a facet of $\mbox{del}_{\Delta}(x)$.

A vertex satisfying condition (2) is called a {\it shedding vertex} of $\Delta$.
\end{Definition}

For the recent developments on vertex decomposability,  one may refer to \cite{GLW} and the references therein. The following strict implications is well-known for a  simplicial complex $\Delta$:
$$ \mbox{vertex decomposable }\Longrightarrow  \mbox{  shellable }\Longrightarrow   \mbox{  sequentially Cohen-Macaulay}$$
Moreover, $\Delta$  is shellable if and only if $I_{\Delta}^{\vee}$ has linear quotients; and $\Delta$  is sequentially Cohen-Macaulay if and only if $I_{\Delta}^{\vee}$ is componentwise linear.  One may ask what property $I_{\Delta}^{\vee}$ has when $\Delta$ is  vertex decomposable, or vice versa.

\begin{Definition} \em Following \cite{KH}, we say that a monomial ideal $I$ in $R$  is {\it weakly polymatroidal} if for every pair of elements $u=x_1^{a_1}\cdots x_n^{a_n}$ and $v=x_1^{b_1}\cdots x_n^{b_n}$ of $G(I)$ with $a_1=b_1,\ldots,a_{q-1}=b_{q-1}$ and $a_q<b_q$, (noting that $q<n$) there exists $p>q$
such that $w:=(x_qu)/x_p$ belongs to $G(I)$. Here, $G(I)$ denotes the set of minimal monomial generators of $I$.
 \end{Definition} Different from  the original definition in \cite{KH}, we here do not require $I$ to be generated in one degree. Using the same method as in \cite{KH}, one can prove if $I$ is weakly polymatroidal  in this generalized sense, then $I$ has linear quotients.
Our second main result is as follows:
\vspace{0.5mm}

\noindent {\bf Theorem~\ref{main2}.} {\it If $I_{\Delta}^{\vee}$ is  weakly polymatroidal   then $\Delta$ is vertex decomposable.}
\vspace{0.5mm}

 This particularly shows that being weakly  polymatroidal  is a condition stronger than the property of having linear quotients for a monomial ideal. We also prove that the converse of  Theorem~\ref{main2} holds in some special cases. Recall that a graph $G$ is  {\it unmixed} if every minimal vertex cover of $G$ has the same cardinality, i.e. $J(G)$ is generated in one degree, and that $G$ is vertex decomposable if the independence complex of $G$ is vertex decomposable.

\vspace{0.5mm}
\noindent {\bf Corollary~\ref{3.2}.} {\it  Let $G$ be either a cactus graph or a bipartite graphs or  a chordal graph. Assume further that $G$ is unmixed. Then the following are equivalent:
\begin{enumerate}
  \item $R/I(G)$ is Cohen-Macaulay;
  \item $G$ is vertex decomposable;
  \item $J(G)$ is weakly polymatroidal (in some ordering  of variables);
   \item $J(G)^s$ is weakly polymatroidal (in some ordering  of variables) for all $s\geq 1$.
\end{enumerate} }

 Here, a {\it cactus} graph is a simple graph in which every edge  belongs to at most one cycle. This leads us to conjecture the following.

   \noindent {\bf Conjecture~\ref{3}.} {\it If $G$ is an unmixed graph, then $G$ is vertex decomposable if and only if $J(G)$ is weakly polymatroidal.}
   \vspace{0.5mm}

 The condition that $G$ is  unmixed cannot be dropped in the above conjecture. We  show that  the above conjecture holds true if either $G$ has girth $\geq 5$ or $G$ is very well-covered, see Propositions~\ref{3.5} and \ref{3.6}.
\vspace{0.2mm}

  It is natural to ask  for which unmixed graphs $G$,  $J(G)^{s}$  are  weakly polymatroidal  besides the graphs given in Corollary~\ref{3.2}. Of course, such  graphs (i.e., $R/I(G)$) should be Cohen-Macaulay by Alexander duality.     Let $G$ be a simple graph on vertex set $V(G)$ with edge set $E(G)$. Following \cite{CN}, a {\it clique vertex-partition} of $G$ is a partition $V(G)=W_1\sqcup \cdots \sqcup W_t$ such that the induced graph of $G$ on $W_i$ is a clique (a complete graph) for $i=1,\ldots,t$.   Denote this partition by $\pi=\{W_1, \ldots,W_t\}$. The {\it fully clique-whiskering graph} $G^{\pi}$ of $G$ by $\pi$  is the graph   on  vertex set $V(G)\cup\{y_1,\ldots, y_t\}$ and with edge set $E(G)\cup\{vy_i \:\; v\in W_i, 1\leq i\leq t\}$. When $\pi$ is a trivial partition, i.e., $|W_1|=\cdots=|W_t|=1$, $G^{\pi}$ is also called the {\it whisker graph} of $G$. A tree is Cohen-Macaulay if and only if it is the whisker graph of some other tree. A Cameron-Walker graph is  Cohen-Macaulay if and only if it is a fully clique-whiskering graph $G^{\pi}$ of a bipartite gaph $G$ by some clique vertex-partition $\pi$ of $G$,  see \cite[Theorem 1.3]{HHKO}. Our third main result is as follows:

 \vspace{1mm}
 \noindent {\bf Theorem~\ref{main3}.} {\it If $W=G^{\pi}$ for some graph $G$ and some clique vertex-partition $\pi$, then   $J(W)^s$ is weakly polymatroidal for all $s\geq 1$.}
  \vspace{1mm}

 This result is a complement of Corollary~\ref{3.2}.   As a consequence, we obtain if $W=G^{\pi}$  then $\reg(J(W)^s)=\reg(J(W)^{(s)})=s|V(G)|$ for all $s\geq 1$. The following is another consequence of Theorem~\ref{main3}.
   \vspace{1mm}

   \vspace{1mm}
 \noindent {\bf Corollary~\ref{main3c}.} {\it If $W$  is the whisker graph of  some graph, then  both $J(W)^s$ and $J(W)^{(s)}$ are weakly polymatroidal for all $s\geq 1$.}
  \vspace{1mm}

    In the rest part of this paper  we will keep the notions introduced  in this section unless otherwise stated, and refer to \cite{HH} for some unexplained notions.

\section{Powers of cover ideals of odd cycles}

In this section we will prove that if $C$ is an odd cycle then both $J(C)^{s}$ and $J(C)^{(s)}$ have the same regularity for all $s\geq 1$.

We begin with fixing some notions. Let $M$ be a finitely generated graded $R$-module generated by homogeneous elements $f_1,\ldots,f_r$ minimally with $\deg(f_1)\leq \deg(f_2)\leq \cdots \leq \deg(f_r)$. We denote  by $\mathrm{Deg}(M)$ the number $\deg(f_r)$ and by $\deg(M)$ the number $\deg(f_1)$. It is known that $\reg(M)\geq \mathrm{Deg}(M).$ Let $\mathfrak{m}$ denote the maximal graded ideal $(x_1,x_2,\ldots,x_n)$ of $R$.

\begin{Proposition} \label{F1}Let $I$ be a  homogeneous ideal in $R$ and $t$ a positive integer. Put $J=I\cap \mathfrak{m}^t$.  Then the following holds.
\begin{enumerate}

\item $\mathrm{reg}(\frac{R}{I})\leq \mathrm{reg}(\frac{R}{J})$;

\item  If $t\leq \mathrm{Deg}(I)$, then $\mathrm{reg}(\frac{R}{I})=\mathrm{reg}(\frac{R}{J}).$

  \end{enumerate}
\end{Proposition}
\begin{proof}

(1) Set $b_i=\max\{j\:\;H_{\mathfrak{m}}^i(\frac{R}{I})_j\neq 0\}$ and $a_i=\max\{j\:\; H_{\mathfrak{m}}^i( \frac{R}{J})_j\neq 0\}$. Then $\mathrm{reg}(\frac{R}{I})=\max\{b_i+i\:\; i\geq 0\}$ and $\mathrm{reg}(\frac{R}{J})=\max\{a_i+i\:\; i\geq 0\}$.  Applying the local cohomological functors with respect  to $\mathfrak{m}$ to the short exact sequence  \begin{equation}\label{1}\tag{$\dag$} 0\rightarrow \frac{R}{J}\rightarrow \frac{R}{I}\oplus \frac{R}{\mathfrak{m}^t}\rightarrow \frac{R}{\mathfrak{m}^t+I}\rightarrow 0,\end{equation}
we obtain the  long exact sequence: $0\rightarrow H_{\mathfrak{m}}^0(R/J)\rightarrow  H_{\mathfrak{m}}^0(\frac{R}{I})\oplus H_{\mathfrak{m}}^0(\frac{R}{\mathfrak{m}^t})\rightarrow H_{\mathfrak{m}}^0(\frac{R}{\mathfrak{m}^t+I})\rightarrow \cdots \rightarrow H_{\mathfrak{m}}^i(R/J)\rightarrow  H_{\mathfrak{m}}^i(\frac{R}{I})\oplus H_{\mathfrak{m}}^i(\frac{R}{\mathfrak{m}^t})\rightarrow H_{\mathfrak{m}}^i(\frac{R}{\mathfrak{m}^t+I})\rightarrow \cdots$. From  this sequence  as well as the equality $H_{\mathfrak{m}}^i(\frac{R}{\mathfrak{m}^t+I})=H_{\mathfrak{m}}^i(\frac{R}{\mathfrak{m}^t})=0$ for all $i>0$, we obtain the following facts.

(i) $a_i=b_i$ for all $i\geq 2$.

(ii) The sequence $H_{\mathfrak{m}}^1(\frac{R}{J})\rightarrow H_{\mathfrak{m}}^1(\frac{R}{I})\rightarrow 0$ is exact. From this, we have $H_{\mathfrak{m}}^1(\frac{R}{I})_i=0$ if $H_{\mathfrak{m}}^1(\frac{R}{J})_i=0$. This implies that $a_1\geq b_1$.

(iii) The sequence $0\rightarrow H_{\mathfrak{m}}^0(\frac{R}{J})\rightarrow H_{\mathfrak{m}}^0(\frac{R}{I})\oplus \frac{R}{\mathfrak{m}^t}\rightarrow \frac{R}{I+\mathfrak{m}^t}$ is exact. Thus, for any $i\in \mathbb{Z}$ with $H_{\mathfrak{m}}^0(\frac{R}{J})_i=0$, we have
$\dim_kH_{\mathfrak{m}}^0(\frac{R}{I})_i+\dim_k [\frac{R}{\mathfrak{m}^t}]_i \leq\dim_k [\frac{R}{I+\mathfrak{m}^t}]_i\leq \dim_k [\frac{R}{\mathfrak{m}^t}]_i$ and so $\dim_kH_{\mathfrak{m}}^0(\frac{R}{I})_i=0$.  Hence $a_0\geq b_0$.

  Combining  (i),(ii) with (iii), we obtain that  $\mathrm{reg}(\frac{R}{J})\geq \mathrm{reg}(\frac{R}{I})$.

(2) By using the short exact sequence (\ref{1}),  we also obtain $$\mathrm{reg}(\frac{R}{J})\leq \max\{\mathrm{reg}(\frac{R}{I}), t-1\}=\mathrm{reg}(\frac{R}{I}).$$ This finishes the proof.
\end{proof}

Proposition~\ref{F1} can be extended  to the case of graded modules. Let $M=\bigoplus_{i\in \mathbb{Z}}M_i$ be a finitely generated $\mathbb{Z}$-graded $R$-module. For  $j\in \mathbb{Z}$, we denote by $M_{\geq j}$  the graded submodule $\bigoplus_{i\geq j}M_i$ of $M$. Note that $I\cap \mathfrak{m}^t=I_{\geq t}$ for any graded ideal $I$, we may look upon the following result as a generalization of Proposition~\ref{F1}.

\begin{Proposition}  Let $M$ be a finitely generated $\mathbb{Z}$-graded $R$-module and let $j\in \mathbb{Z}$ such that $M_{\geq j}\neq 0$. Then the following statements hold:
\begin{enumerate}
\item $\reg(M_{\geq j})\geq \reg(M)$;

\item If $j\leq \mathrm{Deg}(M)$, then $\reg(M_{\geq j})=\reg(M)$.
\end{enumerate}

\end{Proposition}

\begin{proof} (1)  Consider the following short exact sequence: \begin{equation}\label{2} \tag{$\ddag$} 0\rightarrow M_{\geq j}\rightarrow M\rightarrow M/M_{\geq j}\rightarrow 0.\end{equation}
It is known that $\reg(N)=\max\{i\:\; N_i\neq 0\}$ whenever $N$ is a graded $R$-module of finite length. From this, it follows that  $\reg(M_{\geq j})\geq j> \reg(M/M_{\geq j})$. Hence $\reg(M)\leq \max\{\reg(M_{\geq j}), \reg(M/M_{\geq j})\}=\reg(M_{\geq j})$.

(2) Using the sequence (\ref{2}) again, we  obtain    $$\reg(M_{\geq j})\leq \max\{\reg(M), \reg(M/M_{\geq j})+1\}.$$  Note that $\reg(M)\geq \mathrm{Deg}(M)\geq j\geq \reg(M/M_{\geq j})+1$, the result follows.\end{proof}

Let $G$  be a simple graph on vertex set $[n]$ and $H$  a subgraph of $G$. The neighborhood of $H$ is defined by
$$N_G(H)=\{i\in V(G)\:\; i \mbox{ is adjacent to some vertex of } H \}.$$
 By \cite[Proposition 5.3]{HHT},  if $N_G(C)=[n]$ for every odd cycle $C$, then the symbolic Rees algebra $$\mathcal{R}_s(J(G))=\bigoplus_{k\geq 0} J(G)^{(k)}t^k$$ of $J(G)$  is generated by the monomial $x_1\cdots x_nt^2$ together with  the monomials  $t\prod_{i\in F}x_i$ such that $F$  is a minimal vertex cover of $G$. Thus,  the following result is a direct consequence of \cite[Proposition 5.3]{HHT}.
\begin{Proposition} \label{herzog} Let $G$ be  a simple graph on vertex set $[n]$ such that $N_G(C)=[n]$ for every odd cycle of $G$. Then $$J(G)^{(s)}=J(G)^s+\sum_{i=1}^{\lfloor\frac{s}{2}\rfloor}(x_1x_2\cdots x_{n})^iJ(G)^{s-2i}.$$
Here, $\lfloor\frac{s}{2}\rfloor$ denotes the largest integer at most $\frac{s}{2}$.
\end{Proposition}

Let $C$ be an odd cycle of length $n=2r+1$. It is not difficult to see that $C$ is not unmixed if $n\geq 9$. More precisely, we have $\deg(J(C))=r+1$ and $$\mathrm{Deg}(J(C))=\left\{
                                                                                                     \begin{array}{ll}
                                                                                                       4t+2, & \hbox{$n=6t+3$;} \\
                                                                                                       4t+3, & \hbox{$n=6t+5$;} \\
                                                                                                       4t+4, & \hbox{$n=6t+7$.}
                                                                                                     \end{array}
                                                                                                   \right.
$$ for all $t\geq 0$.

We now come to the main result of this section.

\begin{Theorem} \label{main1} Let $J$ be the vertex cover ideal of an  odd cycle of $C$ on  vertex set $[n]$ with $n=2r+1$. Then $\reg(J^{(s)})=\reg(J^s)$ for all $s\geq 1$.

\end{Theorem}

\begin{proof} Put $t=\mathrm{Deg}(J)$. Then $t\geq r+1$ and so $\mathrm{Deg}(J^{(s)})=\mathrm{Deg}(J^s)=st$ for all $s\geq 1$ by Proposition~\ref{herzog}. Fix $s\geq 1$. We claim that $$J^{(s)}\cap \mathfrak{m}^{st}=J^{s}\cap \mathfrak{m}^{st}.$$   In light of  Proposition~\ref{herzog}, it suffices to show that $$(x_1x_2\cdots x_{2r+1})^iJ^{s-2i}\cap \mathfrak{m}^{st} \subseteq J^s$$ for all $i=1,\ldots, \lfloor\frac{s}{2}\rfloor$.
Fix $i\geq 1$ and let $\alpha$ be a  monomial  in  $(x_1x_2\cdots x_{2r+1})^iJ^{s-2i}\cap \mathfrak{m}^{st}$. Then we may write  $\alpha=(x_1x_2\cdots x_{2r+1})^i\alpha_1\cdots \alpha_{s-2i}\mathbf{u}$, where $\alpha_i$ is a minimal monomial generator of $J$ for each $i$ and $\mathbf{u}$ is some monomial. Since $\deg(\alpha)\geq st$ and $\deg(\alpha_i)\leq t$, it follows that $\deg(\mathbf{u})\geq i$. This, together with the fact  $(x_1x_2\cdots x_{2r+1})x_i\in J^2$ for any $i\in [2r+1]$, implies $\alpha\in J^{s}$.  Thus, the claim is proved.

In view of Proposition~\ref{F1}, the result is immediate from the above claim.
\end{proof}

\begin{Remark} \em Proposition \ref{F1} is also useful in the study of powers of edge ideals.  For example,  \cite[Theorem 3.5]{JK} is a direct consequence of \cite[Corollary 3.3]{JK} in view of Proposition \ref{F1}.
\end{Remark}

\section{Vertex decomposability via weak polymatroidality}

Let $\Delta$ be a simplicial complex on $[n]$. In this section we prove that if  $I_{\Delta}^{\vee}$ is weakly polymatroidal then $\Delta$ is vertex decomposable. This particularly shows that for a monomial ideal, the property of being weakly polymatroidal  is stronger than the property of having linear quotients. The converse implication of the above result is also discussed.

We first give an observation on the property of a shedding vertex. Let $\Delta$ be a simplicial complex on vertex set $[n]$ with facets $ F_1,F_2, \ldots,F_r$. Assume that $k$ is a shedding vertex of $\Delta$ and assume  without loss of generality that $k\in F_i$ for $i=1,\ldots,j$ and $k\notin F_i$ for $i=j+1,\ldots,r$. Then $\mbox{lk}_{\Delta}(k)=\langle F_1\setminus \{k\}, \ldots, F_j\setminus \{k\}\rangle$ and  $\mbox{del}_{\Delta}(k)=\langle F_1\setminus \{k\}, \ldots, F_j\setminus \{k\}, F_{j+1}, \ldots, F_r\rangle=\langle F_{j+1}, \ldots, F_r\rangle$. This observation is useful in the following proofs. We also need some more notation. Let $I$ be a monomial ideal.
As usual, $G(I)$ denotes the set of minimal monomial generators of $I$ and $\mathrm{supp}(I)$ is the set $\cup_{u\in G(I)}\mathrm{supp}(u)$, where for a monomial $u$,   $\mathrm{supp}(u)$ denotes the set $\{i\in [n]\:\; x_i| u\}$. For a subset $A\subseteq [n]$, $x_{\overline{A}}$ denotes the monomial $\prod_{i\in [n]\setminus A}x_i$.

\begin{Theorem} \label{main2} Let $\Delta$ be a simplicial complex on $[n]$ and suppose  that $I_{\Delta}^{\vee}$ is weakly polymatroidal in some ordering of   variables. Then  $\Delta$ is vertex decomposable.
\end{Theorem}

\begin{proof} Since the vertex decomposability of a simplicial complex is independent of the ordering of varaibles, we may assume $I_{\Delta}^{\vee}$ is weakly polymatroidal itself. Let $\mathcal{F}(\Delta)$ denote the set of facets of $\Delta$. In the following, we will  use the induction on $|\mathcal{F}(\Delta)|$, the number of facets of $\Delta$. If $|\mathcal{F}(\Delta)|=1$,  then $I_{\Delta}^{\vee}$ is generated by a single monomial and so it is weakly polymatroidal  automatically. Now assume $|\mathcal{F}(\Delta)|\geq 2$. In this case we let $k=\min\{i\:\; i\in F_1\cup\cdots\cup F_r, i\notin F_1\cap\cdots\cap F_r\}$.

We first show that $k$ is a shedding vertex of $\Delta$.    For this, let $F,G$ be facets of $\Delta$ such that $k\in F$ and $k\notin G$, respectively.
Note that we may write $x_{\overline{F}}=x_{k+1}^{a_{k+1}}\cdots x_n^{a_n}$ and $x_{\overline{G}}=x_kx_{k+1}^{b_{k+1}}\cdots x_n^{b_n}$, where $a_i,b_i\in \{0,1\}$ for all $i$ and they are minimal generators of $I_{\Delta}^{\vee}$.   This implies there exists $\ell>k$ such that  $\mathbf{u}:=x_kx_{\overline{F}}/ x_{\ell}$ is also a minimal generators of $I_{\Delta}^{\vee}$ and so there exists a facet $H$ of $\Delta$ such that  $\mathbf{u}=x_{\overline{H}}$.  From this it follows that $k\notin H$ and $F\setminus \{k\} \subseteq H$. Since $H$ is  a facet of $\Delta$, we have $F\setminus \{k\}\subsetneq H$. This actually shows that none of facets of $\mathrm{lk}_{\Delta}(k)$ is  a facet of $\mathrm{del}_{\Delta}(k)$ and it follows that $k$ is a shedding vertex of $\Delta$.

Set $\Delta_1:=\mathrm{lk}_{\Delta}(k)$
and $\Delta_2:=\mathrm{del}_{\Delta}(k)$. Next we show $I_{\Delta_1}^{\vee}$ and $I_{\Delta_2}^{\vee}$ are both weakly polymatroidal.  We will look $\Delta_1$ and $\Delta_2$ upon as simplicial complexes on $V:=[n]\setminus \{k\}$.
 Then $I_{\Delta_i}^{\vee}=(x_{V\setminus F}\:\; F\in \mathcal{F}(\Delta_i))$.   Note that
 \begin{equation}\label{D}\tag{$\clubsuit$}
 G(I_{\Delta}^{\vee})=G(I_{\Delta_1}^{\vee})\bigsqcup \{x_ku\:\; u\in G(I_{\Delta_2}^{\vee})\}.
 \end{equation}
 Moreover, $\mathrm{supp}(I_{\Delta_i}^{\vee})\subseteq \{k+1,\ldots,n\}$ for $i=1,2$.

 Let $\mathbf{u}=x_{k+1}^{a_{k+1}}\cdots x_n^{a_n}$ and $\mathbf{v}=x_{k+1}^{b_{k+1}}\cdots x_n^{b_n}$ be distinct elements in $G(I_{\Delta_1}^{\vee})$ such that $a_{k+1}=b_{k+1},\ldots,a_{k+i-1}=b_{k+i-1}$ and  $a_{k+i}<b_{k+i}$. Then, since $I_{\Delta}^{\vee}$ is weakly polymatroidal, there exists $j>i\geq 1$ such that $\mathbf{w}:=x_{k+i}\mathbf{u}/x_{k+j}\in G(I_{\Delta}^{\vee})$. This, together with the decomposition in (\ref{D}), implies $\mathbf{w}\in G(I_{\Delta_1}^{\vee})$. Thus, we have proven that $I_{\Delta_1}^{\vee}$ is weakly polymatroidal. Similarly, we can prove $I_{\Delta_2}^{\vee}$ is weakly polymatroidal.
 By induction hypothesis, we have $\Delta_i$ is vertex decomposable for $i=1,2$ and so $\Delta$ is vertex decomposable, as required.
\end{proof}

Let $G$ be a simple graph. A subset $A$ of $V(G)$ is an {\it independent  set} of $G$ if for any $i,j\in A$, the pair $\{i,j\}$ is not an edge of $G$. The {\it independence complex} of $G$ is the collection of all independent  sets of $G$. We call $G$ to be vertex decomposable if its independence complex is vertex decomposable. If we let $\Delta$ be the  independence complex of $G$, then the Stanley-Reisner ideal $I_{\Delta}$ is the edge ideal $I(G)$ and its Alexander dual  $I_{\Delta}^{\vee}$ is the vertex cover ideal $J(G)$.

The converse of Theorem~\ref{main2} is true in some cases as shown by the following corollary.

\begin{Corollary} \label{3.2} Let $G$ be either a cactus graph or a bipartite graphs or  a chordal graph. Assume further that $G$ is unmixed. Then the following are equivalent:

\begin{enumerate}
  \item $R/I(G)$ is Cohen-Macaulay;
  \item $G$ is vertex decomposable;
  \item $J(G)$ is weakly polymatroidal (in some ordering  of variables);
   \item $J(G)^s$ is weakly polymatroidal (in some ordering  of variables) for all $s\geq 1$.
\end{enumerate}

 In particular, if $\Delta$ is the independence complex of either a  cactus graph or a  bipartite graph or a chordal graph, then $I_{\Delta}^{\vee}$ is weakly polymatroidal in some ordering of variables if and only if $\Delta$ is vertex decomposable and pure.
\end{Corollary}

\begin{proof}
(1)$\Rightarrow$(4)  Suppose that $R/I(G)$ is Cohen-Macaulay. If $G$ is either a cactus graph or a chordal graph,  then $J(G)^{s}$ is weakly polymatroidal by  \cite[Theorem 4.3]{M} and \cite[Theorem 1.7]{M1} respectively.  If $G$ is bipartite, then $G$ comes from  a finite poset $P$, see \cite[Theorem 9.1.13]{HH}.  Using the notation in \cite{EHM}, we may write  $I(G)=I_2(P)$. From this it follows that  $(I_{\Delta}^{\vee})^s=(H_2(P))^s$ is weakly polymatroidal for all $s\geq 1$ by \cite[Theorem 2.2]{EHM}.

(4)$\Rightarrow$(3) Automatically.

(3)$\Rightarrow$(2) It follows from Theorem~\ref{main2}.

(2)$\Rightarrow$(1) Every vertex decomposable graph is sequentially Cohen-Macaulay and an unmixed sequentially Cohen-Macaulay  graph is a Cohen-Macaulay graph.
\end{proof}

\begin{Corollary} Let $C$ be a cycle of size $5$. Then $\reg(J(C)^{(s)})=\reg(J(C)^s)=3s.$
\end{Corollary}

\begin{proof} It is known that if $C$ is a cycle of size $n$, then $R/I(C)$ is Cohen-Macaulay if and only if $R/I(C)$ is sequentially Cohen-Macaulay if and only if $n\in \{3,5\}$, see \cite[Proposition 4.1]{FT}. By this fact the result follows from Theorem~\ref{main1} together with Corollary~\ref{3.2}.
\end{proof}

We also have the following  partial converse of Theorem~\ref{main2}.

\begin{Proposition}\label{3.4} Let $\Delta$ be a pure simplicial complex on $[n]$. If $1$ is a shedding vertex of $\Delta$ such that $I_{\Delta_i}^{\vee}$ is weakly polymatroidal for $i=1,2$, then $I_{\Delta}^{\vee}$ is weakly polymatroidal. Here, $\Delta_1:=\mathrm{lk}_{\Delta}(1)$
and $\Delta_2:=\mathrm{del}_{\Delta}(1)$.
\end{Proposition}
\begin{proof} Since $1$ is a shedding vertex of $\Delta$, we have \begin{equation*}
 G(I_{\Delta}^{\vee})=G(I_{\Delta_1}^{\vee})\bigsqcup \{x_1u\:\; u\in G(I_{\Delta_2}^{\vee})\}.
 \end{equation*}
 Note that  $\mathrm{supp}(u)\subseteq \{2,\ldots,n\}$ for any $u\in G(I_{\Delta_1}^{\vee})\bigsqcup  G(I_{\Delta_2}^{\vee})$.

 Let $\mathbf{u}=x_1^{a_1}\cdots x_n^{a_n}$ and $\mathbf{v}=x_1^{b_1}\cdots x_n^{b_n}$ be monomials belonging to $G(I_{\Delta}^{\vee})$ satisfying $a_1=b_1,\ldots,a_{i-1}=b_{i-1}$ and $a_i<b_i$. We need to find a monomial $\mathbf{w}\in G(I_{\Delta}^{\vee})$ such that $\mathbf{w}=x_i\mathbf{u}/x_j$ for some $j>i$.  If either $\{\mathbf{u,v}\}\subseteq G(I_{\Delta_1}^{\vee})$ or  $\{\mathbf{u,v}\}\subseteq \{x_1u\:\; u\in G(I_{\Delta_2}^{\vee})\}$, the existence of $\mathbf{w}$ follows by the assumption that $I_{\Delta_i}^{\vee}$ is weakly polymatroidal for $i=1,2$. So we only need to consider the case that $\mathbf{u}\in G(I_{\Delta_1}^{\vee})$ and $\mathbf{v}\in \{x_1u\:\; u\in G(I_{\Delta_2}^{\vee})\}$. Note that $i=1$ in this case.  Moreover, since $1$ is a shedding vertex, there exists  $\mathbf{v_1}\in G(I_{\Delta_2}^{\vee})$ such that $\mathbf{u}$ divides $\mathbf{v_1}$.   This implies $\mathbf{v_1}=\mathbf{u}x_j$ for some $j>1$ by the purity of $\Delta$. Hence $\mathbf{w:}=  x_1\mathbf{v_1}=x_1\mathbf{u}/x_j$ meets the requirement and so we are done. \end{proof}

 Proposition~\ref{3.4} together with Corollary~\ref{3.2} leads us to present the following.

\begin{Conjecture} If $\Delta$ is a vertex decomposable  pure simplicial complex, then $I_{\Delta}^{\vee}$ is weakly polymatroidal.
\end{Conjecture}

A weak form of this conjecture is:

\begin{Conjecture} \label{3} If $G$ is an unmixed vertex decomposable graph, then $J(G)$ is weakly polymatroidal.
\end{Conjecture}

The condition that $G$ is unmixed is necessary in the above conjecture as shown by the following example.
\begin{Example}\em A {\it star graph} is a graph in which exactly one vertex has degree at least 2.
Let $G$ be a star graph with more than three vertices of degree 1. Then $G$ is vertex decomposable, but $J(G)$ is not  weakly polymatroidal in any ordering of variables.

\end{Example}

The following results are other examples in support of Conjecture 2.
\begin{Proposition} \label{3.5} Let $G$ be a Cohen-Macaulay graph of girth at least 5. Then $J(G)$ is weakly polymatroidal.
\end{Proposition}
\begin{proof} By \cite[Theorem 2.4]{HMT},  we have $G$ is a $\mathcal{PC}$ graph. Recall an induced 5-cycle of  $G$ is  {\it basic} if it does not contains two adjacent vertices of degree three or more in $G$, and an edge is  a {\it pendant} if it contains a vertex of degree 1.   Let $C^1, \ldots, C^k$ be the set of all basic 5-cycles of $G$ and let $L^1,\ldots,L^l$ be the set of all pendants of $G$.  Recall that $G$ is a $\mathcal{PC}$ graph if $V(G)$ can be partitioned into $V(G)=V(C^1)\sqcup \cdots \sqcup V(C^k)\sqcup  V(L^1)\sqcup \cdots \sqcup V(L^l)$.

Label the  vertices of $C^{i}$ successively by $x_{i1}, x_{i4}, x_{i2}, x_{i3}, x_{i5}$ such that $x_{i3},x_{i4}, x_{i5}$ are vertices of degree 2 for $i=1,\ldots,k$, and let $y_{i1}, y_{i2}$ be vertices of $L^i$ such that $y_{i2}$  has degree 1 for $i=1,\ldots,l$. We work with the following ordering of variables:
$$x_{11}>\cdots >x_{15}>\cdots>x_{k1}>\cdots x_{k5}>y_{11}>y_{12}>\cdots>y_{l1}>y_{l2}.$$
To prove $J(G)$ is weakly polymatroidal, we let $f,g$ be minimal monomial generators of $J(G)$ with $f\neq g$, and let $z$ be a variable such that $\deg_{z'}f=\deg_{z'}g$ for $z'>z$ and $\deg_{z}g<\deg_{z}f$. We need to find a variable $w<z$ such that $zg/w\in J(G)$. To this end, for each monomial $h\in J(G)$,  we write $h$ as following: $$h=h(C^1)\cdots h(C^k)h(L^1)\cdots h(L^l),$$ where $h(C^i)$ and $h(L^j)$ are monomials such that $\mathrm{supp}(h(C^i))\subseteq V(C^i)$ for $i=1,\ldots,k$ and $\mathrm{supp}(h(L^j))\subseteq V(L^j)$ for $j=1,\ldots,l$.   We consider the following cases:

{\bf Case 1} $z=x_{i1}$ for some $i$: Then $g(C^i)\in \{x_{i2}x_{i4}x_{i5}, x_{i3}x_{i4}x_{i5}\}$. We set $w=x_{i4}$ if $g(C^i)=x_{i2}x_{i4}x_{i5}$ and  set $w=x_{i5}$ if $g(C^i)=x_{i3}x_{i4}x_{i5}$. Then $zg/w\in J(G)$.

{\bf Case 2} $z=x_{i2}$ for some $i$: Then $g(C^i)\in \{x_{i1}x_{i4}x_{i3}, x_{i3}x_{i4}x_{i5}\}$. We set $w=x_{i4}$  if $g(C^i)=x_{i1}x_{i4}x_{i3}$, and set $w=x_{i3}$ and  set $w=x_{i4}$ if $g(C^i)=x_{i3}x_{i4}x_{i5}$. Then $zg/w\in J(G)$.

{\bf Case 3} $z=x_{i3}$ for some $i$: Then $g(C^i)=x_{i1}x_{i5}x_{i2}$. We set $w=x_{i5}$.

 {\bf Case 4} $z=x_{i4}$ for some $i$: This case is impossible.

{\bf Case 5} $z=x_{i5}$ for some $i$: This case is impossible again.

 {\bf Case 6} $z=y_{i1}$ for some $i$: Then $h(L_i)=y_{i2}$. We set $w=y_{i2}$.

 {\bf Case 7} $z=y_{i2}$ for some $i$:  This case is  impossible again.

 Thus, in all possible cases, we find a variable  $w$ which meets the requirement. This completes the proof.\end{proof}

    A graph $G$ is called  {\it very well-covered} if $|V(G)|$ is even and every minimal vertex cover of $G$ contains exactly $\frac{|V(G)|}{2}$ vertices. Clearly, every very well-covered graph is unmixed, i.e. well-covered.
 \begin{Proposition} \label{3.6} Let $G$ be a Cohen-Macaulay very well-covered graph. Then $J(G)$ is weakly polymatroidal.
\end{Proposition}

\begin{proof} In view of \cite[Lemma 3.1]{MMCRTY}, there is a relabeling of vertices $$V(G)=\{x_1,\ldots,x_n,y_1,\ldots,y_n\}$$ such that the following five conditions hold:

(i) $X=\{x_1,\ldots,x_n\}$ is a minimal vertex cover of $G$ and $Y=\{y_1,\ldots,y_n\}$ is an independent set of $G$;

(ii) $\{x_i,y_i\}\in E(G)$ for $i=1,\ldots,n$;

(iii) if $\{z_i, x_j\}, \{y_j,x_k\}\in E(G)$, then $\{z_i,x_k\}\in E(G)$ for distinct $i,j,k$ and for $z_i\in \{x_i,y_i\}$;

(iv) if $\{x_i,y_j\}\in E(G)$, then $\{x_i,x_j\}\notin E(G)$;

(v) if $\{x_i,y_j\}\in E(G)$, then $i\leq j$.

We will show that $J(G)$ is weakly polymatroidal in the ordering: $x_1>x_2>\cdots>x_n>y_1>y_2>\cdots>y_n$. Let $f,g$ be minimal monomial generators of $J(G)$ with $f\neq g$. Then we may write $f$ and $g$ as follows: $$f=\prod_{z\in C}z=x_1^{a_1}\cdots x_n^{a_n}y_1^{a_{n+1}}\cdots y_n^ {a_{n+n}}$$ and $$g=\prod_{z\in D}z=x_1^{b_1}\cdots x_n^{b_n}y_1^{b_{n+1}}\cdots y_n^{b_{n+n}}.$$
Here, $C$ and $D$ are minimal vertex covers of $G$, and $a_i,b_i\in \{0,1\}$ for all $i$. Note that $a_i+a_{n+i}=1$ and $b_i+b_{n+i}=1$ for $i=1,\ldots,n$, there exists $1\leq k\leq n$ such that $a_i=b_i$ for $i=1,\ldots, k-1$ and $a_k>b_k$. This actually means that $x_k\in C$ and $x_k\notin D$. From this it follows that $y_k\in D$. We claim that $A:=(D\cup \{x_k\})\setminus \{y_k\}$ is also a vertex cover of $G$.

Let $e=\{z_1,z_2\}\in E(G)$. If $y_k\notin e$, then it is clear that $A\cap e\neq \emptyset$. So we assume that $e=\{x_j,y_k\}$. By (v), we have $j\leq k$. If $j=k$, then $A\cap e=\{x_k\}$. So we may assume further that $j<k$.  Since $y_k\notin C$, it follows that $x_j\in C$ and so $b_j=a_j=1.$ This implies $x_j\in D$ and so $A\cap e\neq \emptyset$. Thus, we prove the claim and it follows that $x_kg/y_k=\prod_{z\in A}$ is a minimal monomial generator of $J(G)$, completing the proof.
\end{proof}

Let $G$ be a graph and $s\geq 1$ an integer. In \cite{S2}, the graph $G_s$ is constructed so that $J(G_s)$ is the polarization of $J(G)^{(s)}$.

\begin{Corollary} \label{3.7} Let $G$ be a Cohen-Macaulay very well-covered graph. Then $J(G)^{(s)}$ is weakly polymatroidal for all $s\geq 1$.
\end{Corollary}

\begin{proof} It is immediate from the definitions that if $I$ is a monomial ideal generated in one degree then $I$ is weakly polymatroidal if and only if the polarization of $I$ is weakly polymatroidal. By \cite[Proposition 3.1]{S2}, if $G$ is a Cohen-Macaulay very well-covered graph then so is $G_s$. Now, the result follows from Proposition~\ref{3.6}.
\end{proof}

\section{Powers of cover ideals of clique-whiskering graphs}

Let $W=G^{\pi}$ be the fully clique-whiskering of some graph $G$ by some clique vertex-partition $\pi$ of $G$. In this section we  prove  $J(W)^s$ is weakly polymatroidal for all $s\geq 1$.  As a consequence, we have  $\reg(J(W)^s)=\reg(J(W)^{(s)})=s|V(G)|$ for all $s\geq 1$.

For convenience, we introduce  the notions of a simplicial co-complex and the face ideal of a simplicial co-complex with respect to a partition.
\begin{Definition} \label{4.1}\em Let $V$ be a finite set.  We say that a collection $\nabla$ of subsets of $V$ is a {\it simplicial co-complex} on $V$  if whenever $F\in \nabla$ and $F\subseteq G\subseteq V$ one has $G\in \nabla$. Assume that $V$ has a partition $V=V_1\sqcup \cdots\sqcup V_t$ and assume that $V_i=\{x_{i1},\ldots,x_{ik_i}\}$ for $i=1,\ldots, t$. For each face $F\in \nabla$, we put $$u_F:=\prod_{x_{ij}\in F}x_{ij}\prod_{i\in [t]}y_i^{k_i-|F\cap V_i|}.$$ Then the face ideal of $\nabla$ with respect to this partition is the following monomial ideal $J$ in the polynomial ring $k[x_{11},\ldots,x_{1k_1},\ldots,x_{t1},\ldots,x_{tk_t},y_1,\ldots,y_t]$:
$$J=(u_F:F\in \nabla).$$

\end{Definition}

\begin{Proposition}\label{w1} Let $\nabla$ be  a simplicial co-complex on $V$ and assume   that $V$ has a partition $V=V_1\sqcup \cdots\sqcup V_t$ with $V_i=\{x_{i1},\ldots,x_{ik_i}\}$ for $i=1,\ldots, t$. Denote by $J$ the face ideal of $\nabla$ with respect to this partition.  Then $J^s$ is weakly polymatroidal for each $s\geq 1$ in the  ordering: $x_{11}>\cdots>x_{1k_1}>\cdots>x_{t1}>\cdots>x_{tk_t}>y_1>\cdots>y_t$.  \end{Proposition}

\begin{proof} Let $\alpha,\beta$ be the minimal monomial generators of $J^s$ with $\alpha\neq \beta$. We may write $$\alpha=u_{F_1}u_{F_2}\cdots u_{F_s}=u_1\cdots u_ty_1^{sk_1-\deg(u_1)}\cdots y_t^{sk_t-\deg(u_t)},$$ and $$\beta=u_{G_1}u_{G_2}\cdots u_{G_s}=v_1\cdots v_ty_1^{sk_1-\deg(v_1)}\cdots y_t^{sk_t-\deg(v_t)}.$$

Here, $F_i\in \nabla, G_i\in \nabla$ for all $i\in [s]$, and $\mathrm{supp}(u_i)\cup \mathrm{supp}(v_i)\subseteq V_i$ for $i=1,\ldots,t$. Since $\alpha\neq \beta$, there exists a variable $z<y_1$ such that $\deg_{w}(\alpha)=\deg_{w}(\beta)$ for all $w<z$ and $\deg_{z}(\alpha)<\deg_{z}(\beta)$. Suppose that $z=x_{ij}$, where $1\leq i\leq t$ and $1\leq j\leq k_i$. Then $\deg_{x_{ij}}\alpha\leq s-1$  and so there exists $l\in [s]$,  such that $x_{ij}\notin F_l\cap V_i$. Say $l=1$. Put $F_1':=F_1\cup \{x_{ij}\}$ and let $$\gamma:=u_{F_1'}u_{F_2}\cdots u_{F_s}.$$
Then $\gamma=x_{ij}\alpha/ y_i$ and $\gamma\in J^s$.  From this it follows that $J^s$  is weakly polymatroidal.\end{proof}

\begin{Theorem} \label{main3} Let $W=G^{\pi}$,  where $G$ is a simple graph and $\pi=(W_1,\ldots,W_t)$ is a clique vertex-partition of $G$. Then $J(W)^{s}$ is weakly polymatroidal for every $s\geq 1$.
\end{Theorem}

\begin{proof}  Let $\nabla:
=\{C\cap V(G)\:\;  C\mbox{ is a minimal vertex cover of } W\}.$  Then $\nabla$ is a simplicial co-complex on $V(G)$. Moreover, $J(W)$ coincides with the face ideal of $\nabla$ with respect to the partition $\pi$. Now, the result follows from Proposition~\ref{w1}.
\end{proof}

  Theorem~\ref{main3} together with Theorem~\ref{main2}  implies $G^\pi$ is vertex decomposable.   This recovers \cite[Theorem 3.3]{CN}.  Other consequences of Theorem~\ref{main3} are as follows:

\begin{Corollary} Let $W=G^{\pi}$ be as in Theorem~\ref{main2}. Then $\reg(J(W)^{(s)})=\reg(J(W)^{s})=s|V(G)|$  for all $s\geq 1$.
\end{Corollary}
\begin{proof} It follows from Theorem~\ref{main3} that $\reg(J(W)^{s})=s|V(G)|$.  By \cite[Corollary 4.4]{Sel}, we have $\reg(J(W)^{(s)})=s|V(G)|$.
\end{proof}

\begin{Corollary} \label{main3c} Let $W$ be the whisker graph of some graph. Then both $J(W)^{(s)}$ and $J(W)^s$ are weakly polymatroidal for all $s\geq 1$.
\end{Corollary}
\begin{proof} We have that $J(W)^s$ is weakly polymatroidal by Theorem~\ref{main3} and that $J(W)^{(s)}$ is  weakly polymatroidal by Corollary~\ref{3.7}.
\end{proof}

\vspace{5mm}
{\bf \noindent Acknowledgment:}
This research is supported by NSFC (No. 11971338). We would like to express our sincere thanks to  the referee, who reads the paper carefully and helps us kill many minor errors of the original paper.

\end{document}